\documentclass[12pt,reqno]{amsart}
\usepackage{amsmath,amssymb}
\def\pbla{1.5_\lambda}
\def\pblaest{1.5_{\lambda^*}}

\newcommand{\R}{\mathbb{R}}

\newtheorem{theorem}{Theorem}[section]

\newtheorem{proposition}[theorem]{Proposition}

\newtheorem{remark}[theorem]{Remark}

\numberwithin{equation}{section}

\begin{document}

\title[Regularity of minimizers up to dimension four]
{Regularity of minimizers of semilinear elliptic problems 
up to dimension four}

\author{Xavier Cabr{\'e}}
\thanks{The author was supported by
grants MTM2008-06349-C03-01 (Spain) and 2009SGR345 (Catalunya)}
\address{ICREA and Universitat Polit{\`e}cnica de Catalunya,
Departament de Matem{\`a}-tica Aplicada I, Diagonal 647, 08028
Barcelona, Spain}
\email{xavier.cabre@upc.edu}

\begin{abstract}
We consider the class of semi-stable solutions
to semilinear equations $-\Delta u=f(u)$ in a bounded smooth
domain $\Omega$ of $\R^n$ (with $\Omega$ convex in some results). 
This class includes all
local minimizers, minimal, and extremal solutions.
In dimensions $n\le 4$, we establish an priori $L^\infty$ bound 
which holds for every positive semi-stable solution and every nonlinearity~$f$.
This estimate leads to the boundedness of all extremal solutions
when $n=4$ and $\Omega$ is convex. This result was previously known
only in dimensions $n\le 3$ by a result of G.~Nedev. 
In dimensions $5\leq n\leq9$ the boundedness
of all extremal solutions remains an open question. It
is only known to hold in the radial case $\Omega=B_R$ by 
a result of A.~Capella and the author.
\end{abstract}

\maketitle

\section{Introduction and results}

Let $f:\R\to\R$ be a $C^\infty$ function and $F$ a primitive of $f$, i.e. $F'=f$.
Let $\Omega\subset \mathbb{R}^n$ be a bounded $C^\infty$ domain.
Consider the energy functional
\begin{equation}
 E(u)=\int_\Omega \frac{1}{2}|\nabla u|^2 -F(u)\ dx.
\label{eq:energy}
\end{equation}
Its Euler-Lagrange equation, under zero boundary conditions, is given by
\begin{equation}\label{problem}
\left\{
\begin{array}{rcll}
-\Delta u &=&f(u)&\textrm{in }\Omega\\
u&=&0&\textrm{on }\partial \Omega.
\end{array}\right.
\end{equation}

We say that a function $u\in C^1_0(\overline{\Omega})$ (i.e.,
a $C^1(\overline{\Omega})$ function vanishing on $\partial\Omega$)
is a {\it local minimizer} of \eqref{eq:energy} if there 
exists $\varepsilon>0$ such that
$$
E(u)\le E(u+\xi) 
$$
for every $C^1_0(\overline{\Omega})$ function $\xi$ with 
$\Vert\xi\Vert_{C^1(\overline{\Omega})}\leq \varepsilon$.
By elliptic regularity, every local minimizer $u$ is a $C^\infty$
classical solution of \eqref{problem}. In addition, it is a 
{\it semi-stable} solution in the following sense.

We say that a classical solution $u\in C^2(\overline{\Omega})$ of \eqref{problem}
is {\it semi-stable} if
\begin{equation}
Q_u(\xi ):=\int_{\Omega} \left|\nabla \xi \right|^{2}
-f'(u)\xi^{2}\ dx \geq 0
\label{eq:02}
\end{equation} 
for every $\xi\in C_0^1(\overline{\Omega})$. 
Note that $Q_u$ is the second variation of energy at $u$. The 
semi-stability of a solution $u$ is equivalent to the condition
$\lambda_1\geq0$, where 
$\lambda_1=\lambda_1\left(-\Delta-f'(u);\Omega\right)$ is the first
Dirichlet eigenvalue of the linearized operator $-\Delta-f'(u)$ at
$u$ in $\Omega$. We use the name {\it semi-stable} to distinguish 
from the notion of {\it stable} solution, defined by 
$\lambda_1\left(-\Delta-f'(u);\Omega\right)>0$.
Note that a local minimizer is always semi-stable, but not necessarily
a stable solution.

The following is our main estimate. It originates from questions
raised by H.~Brezis during the nineties (see \cite{B}) on certain 
extremal solutions described below.
In dimensions $n\leq 4$,
we bound the $L^{\infty}(\Omega)$ norm of every positive semi-stable solution $u$ 
by the $W^{1,4}$ norm of $u$ on the set $\left\{u<t\right\}$ 
---where $t$ can be chosen arbitrarily. 
The estimate holds in every smooth domain $\Omega$ (not necessarily convex)
and for every nonlinearity $f$ ---the estimate is indeed completely independent of~$f$.
The importance of the bound is that, by choosing $t$ small, 
$\left\{u<t\right\}$ becomes a 
small neighborhood of $\partial\Omega$, and thus the boundedness of $u$ in 
$\Omega$ is reduced to a question on the regularity of $u$ near 
$\partial\Omega$.

\begin{theorem}\label{Thm1}
Let $f$ be any $C^\infty$ function and
$\Omega\subset \mathbb{R}^n$ any $C^\infty$ bounded domain.
Assume that $2\leq n\leq 4$. 

Let $u\in C^1_0(\overline{\Omega})$, with $u>0$ in $\Omega$,
be a local minimizer of \eqref{eq:energy},
or more generally a classical semi-stable solution of \eqref{problem}.
Then, for every $t>0$,
\begin{equation}\label{est4}
 \Vert u\Vert_{L^\infty(\Omega)} \leq t + \frac{C}{t}|\Omega|^{(4-n)/(2n)}
\left( \int_{\{u <t\}} |\nabla u|^4 \ dx\right)^{1/2},
\end{equation}
where $C$ is a universal constant {\rm (}in particular, independent of 
$f$, $\Omega$,~and~$u${\rm)}. In the last integral we have used the 
notation $\{u <t\}=\{x\in\Omega\; : \; u(x) <t\}$.
\end{theorem}

We will be able to control the right hand side of \eqref{est4}, 
for $t$ small enough, if $\Omega$ is convex. The reason is 
that in convex domains, the moving planes method leads to boundary 
estimates for all positive solutions and all nonlinearities~$f$.

Note that \eqref{est4} is invariant under dilations of the
domain $\Omega$, and also allows to multiply $u$ by a constant 
(by chosing $t$ to have the same units as $u$).
 
The following is the main application of Theorem \ref{Thm1}. 
It motivated this work. Consider the problem
$$
\left\{ \begin{array}{rcll}
 -\Delta u  & = & \lambda g(u) & \textrm{in }\Omega\\
         u  & \ge & 0 & \textrm{in }\Omega\\
         u  & = & 0 & \textrm{on }\partial\Omega,
\end{array}\right.
\eqno{(\pbla)}\stepcounter{equation}
$$
where $\Omega\subset\mathbb{R}^n$ is a smooth bounded domain,
$n\geq 2$, $\lambda \ge 0$, and the nonlinearity 
$g:[0,+\infty)\to\mathbb{R}$ satisfies
\begin{equation}
g\mbox{ is } C^1, \text{nondecreasing, } 
g(0)>0, \textrm{ and }\lim _{u\to +\infty }\frac{g(u)}{u}=+\infty .
\label{nonlg}
\end{equation}
It is well known (see \cite{B,BV,C} and references therein) 
that there exists an
extremal parameter $\lambda^*\in (0,\infty)$ such that if $0\le\lambda<\lambda^*$ then
$(\pbla)$ admits a minimal classical solution $u_\lambda$. In addition,
this solution $u_{\lambda}$ is semi-stable ---see Remark \ref{rem-semi} 
below. On the other hand,
if $\lambda>\lambda^*$ then $(\pbla)$ has no classical solution.
Here, classical means bounded, while minimal means smallest. 
The set $\left\{ u_{\lambda }:0\le \lambda <\lambda ^{*}\right\}$ 
is increasing in $\lambda$ and its limit as 
$\lambda\nearrow\lambda^{*}$ is a weak solution $u^*=u_{\lambda^*}$ 
of $(\pblaest)$, called the extremal solution of $(\pbla)$.
Later we will give the precise meaning of weak solution.

When $g(u)=e^u$, it is known that $u^*\in L^\infty(\Omega)$ 
if $n\leq 9$ (for every~$\Omega$), 
while $u^{*}(x)=-2\log |x|$ if $n\ge 10$ and $\Omega=B_1$.
A similar phenomenon happens when $g(u)=(1+u)^p$ with $p>1$. 
These results date from the seventies. In the nineties important progress 
in the subject came from works of H.~Brezis and collaborators.
He raised the question (see Brezis \cite{B} and also 
Brezis and V\'azquez \cite{BV}) of determining
the regularity of $u^*$, depending
on the dimension~$n$, for general convex nonlinearities 
$g$ satisfying \eqref{nonlg}.
The best known result was established in 2000 by 
G.~Nedev~\cite{ND}. He proved that, for every domain $\Omega$ and convex 
nonlinearity~$g$ satisfying \eqref{nonlg}, 
$u^*\in L^\infty(\Omega)$ if $n\leq 3$, while
$u^*\in H^1_0(\Omega)$ if $n\leq 5$.

In this article, using Theorem \ref{Thm1} we establish the 
boundedness of $u^*$ in dimensions $n\leq 4$ when $\Omega$ is a 
convex domain. We only assume \eqref{nonlg} on the nonlinearity $g$.
In dimension 2, the convexity of $\Omega$ is not assumed.

\begin{theorem}\label{Thm2}
Let $g$ satisfy \eqref{nonlg} and $\Omega\subset\R^{n}$
be a $C^\infty$ bounded domain. Assume that $2\leq n\leq 4$,
and that $\Omega$ is convex in case $n\in \{3,4\}$. 

Let $u^*$ be the extremal solution of $(\pbla)$. Then,
$u^* \in L^\infty(\Omega)$.
\end{theorem}

The validity of this result in nonconvex domains is an open question.
 
The boundedness of $u^{*}$ in dimensions $5\leq n\leq 9$ 
remains an open problem. 
However, in the radial case
$\Omega=B_1$, Capella and the author \cite{CC} have 
established that $u^{*}\in L^{\infty}(B_1)$ whenever $n\leq9$, 
for all nonlinearities $g$ satisfying \eqref{nonlg}. See 
also \cite{CC} for precise pointwise bounds for $u^{*}$ 
in the radial case, and Villegas \cite{Vi} for 
improvements of some of them. These pointwise bounds hold 
not only for extremal solutions $u^{*}$ but for all 
semi-stable solutions.

For $n \in \left\{2,3\right\}$ and general $\Omega$, Nedev 
\cite{ND} assumed $g$ to be convex (a classical hypothesis in the
literature of the subject) and to satisfy \eqref{nonlg}. 
Our result holds also for nonconvex $g$, but when $n=3$ we 
need to assume $\Omega$ to be convex.

In \cite{CS}, Sanch{\'o}n and the author use some ideas of 
the present paper to establish new 
$L^{q}\left(\Omega\right)$ estimates for semi-stable and 
extremal solutions of $-\Delta u=f(u)$ in dimensions 
$n\geq 5$ for general nonlinearities $f$, as well as
related bounds for semilinear equations
involving the fractions of the Laplacian.

The articles \cite{CCS,DDM,S} contain extensions of the 
previous results on extremal solutions for more general 
quasilinear problems and problems related to nonlinear 
Neumann boundary conditions. 

Theorem \ref{Thm2} on the boundedness of $u^{*}$ will follow 
easily from our main estimate \eqref{est4} of Theorem 
\ref{Thm1}. The key point is that the minimal solutions 
$u_{\lambda}$ of $(\pbla)$ are all 
semi-stable (see next remark). The boundary estimate needed 
to control the right hand side of \eqref{est4} is known to 
hold in convex domains ---it follows from the moving planes 
method. Thus, we only assume the convexity of $\Omega$ to 
guarantee boundary estimates for the solution.

\begin{remark}\label{rem-semi}
{\rm
That the minimal solutions $u_{\lambda}$ of 
$(\pbla)$ are semi-stable can be easily 
seen as follows. Since there is no solution smaller than 
$u_{\lambda}$, we have that $u_{\lambda}$ must coincide with 
the absolute minimizer of the energy in the close convex 
set of functions lying between 0 (a strict subsolution) and 
$u_{\lambda}$, a (super) solution. That is, $u_{\lambda}$ 
is a one-sided minimizer (by below). Now, consider small 
perturbations $u_{\lambda}+\varepsilon\xi$, with $\xi\leq0$, 
lying in this convex set. Since $u_{\lambda}$ is a critical 
point of the energy and minimizes it for these perturbations, 
we deduce that the second variation 
$Q_{u_{\lambda}}\left(\xi\right)\geq0$ for all 
$\xi\in C_{0}^{1}\left(\overline{\Omega}\right)$ with $\xi\leq0$ 
in $\Omega$. Now, writting any $\xi$ as 
$\xi=\xi^{+}-\xi^{-}$ (its positive and negative parts) and 
using the form of $Q_{u_{\lambda}}\left(\xi\right)$ and 
that $\xi^{+}$ and $\xi^{-}$ have disjoint supports, 
we conclude $Q_{u_{\lambda}}\left(\xi\right)\geq0$ for all 
$\xi\in C_{0}^{1}\left(\overline{\Omega}\right)$, as claimed. 

Let us mention that the solutions 
$\left\{u_{\lambda}:0\leq\lambda\leq\lambda^{*}\right\}$ 
can also be obtained using the implicit function theorem 
(starting from $\lambda=0$) when $g$ is convex (see 
\cite{BV}). This gives an alternative way of proving 
the semi-stability of $u_{\lambda}$. Another possibility 
is to obtain $u_{\lambda}$ by the monotone iteration 
method, with $0<\lambda<\lambda^{*}$ fixed, starting 
from the strict subsolution $u\equiv0$. This is the reason 
why $u_{\lambda}$ is the minimal (or smallest) solution. 
In addition, its semi-stability can also be proved via this 
monotone iteration procedure. 

See \cite{Da}, and references therein, for more 
information on minimal and extremal solutions, also for 
more general operators.
}
\end{remark}

The following result states a more explicit 
$L^{\infty}\left(\Omega\right)$ estimate for semi-stable 
solutions. The estimate depends on certain assumptions on the 
nonlinearity which are more general 
than \eqref{nonlg}. It applies to certain weak solutions.
We recall that $u\in L^1(\Omega)$ is said to be a weak solution of
\eqref{problem} if $f(u)\text{dist}(\cdot,\partial\Omega)\in L^1(\Omega)$
and $\int_\Omega u\Delta\xi+f(u)\xi\ dx =0$ for all $\xi\in C^2(\overline\Omega)$
with $\xi\equiv 0$ on $\partial\Omega$.
Theorem~\ref{Thm2} on extremal solutions will be an easy consequence
of the next result.

\begin{theorem}\label{Thm3}
Let $f$ be any $C^\infty$ function and 
$\Omega\subset \mathbb{R}^n$ any $C^\infty$ bounded domain.
Assume that $2\leq n\leq 4$,
and that $\Omega$ is convex in case $n\in \{3,4\}$. 

Let $u\in L^{1}(\Omega)$ be a positive weak solution of 
\eqref{problem} and suppose that $u$ is the $L^1(\Omega)$ 
limit of a sequence of classical positive
semi-stable solutions of \eqref{problem}. We then have:

{\rm (i)} If $f\geq 0$ in $[0,\infty)$, then $u \in L^\infty(\Omega)$.

{\rm (ii)} Assume that
\begin{equation}\label{assf}
f(s)\geq c_1>0  \text{ and } f(s)\geq \mu s-c_2 \quad
\text{ for all } s\in  [0,\infty),
\end{equation}
for some positive constants $c_1$ and $c_2$ and for some $\mu >
\lambda_1(\Omega)$, where $\lambda_1(\Omega)$ is the first Dirichlet
eigenvalue of $-\Delta$ in $\Omega$.
Then,
\begin{equation}\label{linf}
 \Vert u\Vert_{L^\infty(\Omega)} \leq C\left( \Omega,\mu,
c_1,c_2,\Vert f\Vert_{L^\infty(\,[0,\overline C(\Omega,\mu,c_2)]\,)} \right),
\end{equation}
where $C(\cdot)$ and $\overline C(\cdot)$ are constants depending only on the
quantities within the parentheses.
\end{theorem}

The proof of our main estimate \eqref{est4} starts by writting 
the semi-stability condition \eqref{eq:02} of the solution with the test 
function $\xi$ replaced by $\xi=c\eta$, where 
$\eta\in C_{0}^{1}\left(\overline{\Omega}\right)$ is still arbitrary and $c$ 
is a well chosen function satisfying a certain equation for the 
linearized operator $-\Delta-f'\left(u\right)$. This idea was already 
used to study minimal surfaces (more precisely, minimal cones), where 
one takes $c=\left|A\right|$ (the norm of the second fundamental 
form of the cone). This motivated our study of the radial semilinear case 
\cite{CC} in which we took $c=\partial_r u$, the radial derivative of 
$u$. See \cite{CC2} for more comments on minimal cones and also on 
similar ideas for harmonic maps. 

In the present paper we take $c=\left|\nabla u\right|$, which satisfies
\begin{equation}\label{eq:grad}
\left(\Delta+f'\left(u\right)\right)\left|\nabla u\right|=
\frac{1}{\left|\nabla u\right|} 
\left(\left|\nabla_T\left|\nabla u\right|\right|^2+
\left|A\right|^2\left|\nabla u\right|^2\right)\ \text{in}\ 
\Omega\cap\left\{\left|\nabla u\right|>0\right\}.
\end{equation}
Here $\left|A\right|^2=\left|A\left(x\right)\right|^2$ is the squared norm of 
the second fundamental form of the level set of $u$ passing through a 
given $x\in\Omega\cap\left\{\left|\nabla u\right|>0\right\}$, i.e., the 
sum of the squares of the principal curvatures of such level set. On the 
other hand, $\nabla_T$ denotes the tangential gradient to the level set. 
Thus, \eqref{eq:grad} contains geometrical information of the level sets 
of $u$. This geometrical quantities will appear in the expression for 
the second variation of energy $Q_u\left(\xi\right)$ once we take 
$\xi=\left|\nabla u\right|\eta$ and, at the same time, the presence of 
$f'\left(u\right)$ in \eqref{eq:02} will disappear since the left hand 
side of \eqref{eq:grad} refers to $\Delta+f'\left(u\right)$.

Next, we will take $\eta=\varphi\left(u\right)$, i.e.,
$$\xi=\left|\nabla u\right|\varphi\left(u\right),$$
and we will be lead to a Hardy type inequality (for functions $\varphi$ 
of one real variable) involving two weights containing geometrical 
information of the level sets of $u$ ---see inequality \eqref{semi3}. 

The next crucial point in dimension $n=4$ will be the use of a 
remarkable Sobolev inequality on general hypersurfaces of $\mathbb{R}^n$ 
(in our case the level sets of $u$), which involves the mean curvature 
of the hypersurface but that has a best constant independent of the 
hypersurface ---in fact, depending only on the dimension $n$. This 
Sobolev inequality (Theorem \ref{ThmMS} below) is due to Michel-Simon 
\cite{MS} and Allard \cite{A}. 

As we mention in next section, the use of \eqref{eq:grad} and of 
$\xi=\left|\nabla u\right|\eta\left(u\right)$ in the semi-stability 
condition \eqref{eq:02} was first exploited by Sternberg and Zumbrun 
\cite{SZ1,SZ2} to study semilinear phase transitions problems. Farina 
\cite{Fa}, and later Farina-Sciunzi-Valdinoci \cite{FSV} for more 
general quasi-linear operators, have also used this method to establish 
some Liouville type results. 

The results of this paper were first announced in March 2006 in a seminar 
I gave at  ``Analyse Non Lin\'eaire'' (Laboratoire J.L. Lions, Paris VI).

In next section we prove our main estimate, Theorem \ref{Thm1}.
In section~3 we establish Theorem \ref{Thm3} and, as a 
simple consequence, Theorem \ref{Thm2}.

\section{Proof of the main estimate}

In this section we prove our main estimate ---estimate \eqref{est4} of 
Theorem \ref{Thm1}. For this, we will use the following remarkable 
result.

It is a Sobolev inequality due to Michael and Simon~\cite{MS}
and Allard~\cite{A}. It holds on every compact hypersurface of 
$\mathbb{R}^{m+1}$ without boundary, and its constant is independent 
of the geometry of the hypersurface.

\begin{theorem}[\textbf{Michael-Simon~\cite{MS}, Allard~\cite{A}}]
\label{ThmMS}
Let  $M\subset  \R^{m+1}$ be a $C^\infty$ immersed $m$-dimensional 
compact hypersurface without boundary.

Then, for every $p\in  [1, m)$,
there exists a constant $C = C(m,p)$ depending only on the dimension $m$ 
and the exponent $p$ such that, for every $C^\infty$ function $v : M  \to \R$, 
\begin{equation}\label{MSsob}
\left( \int_M |v|^{p^*} dV \right)^{1/p^*} \leq C(m,p)
\left( \int_ M |\nabla v|^p +  |Hv|^p \ dV \right)^{1/p},
\end{equation}
where $H$ is the mean curvature of $M$ and $p^* = mp/(m - p)$.
\end{theorem}

This inequality is stated in Proposition 5.2 of \cite{M}, where references for 
it and related results are mentioned. In \cite{BZ} (section 28.5.2) it is stated
and proved for $p=1$.

The geometric Sobolev inequality \eqref{MSsob} 
has been used in the PDE literature to obtain estimates for the extinction time 
of some geometric evolution flows ---see, for instance, section F.2 of \cite{E} 
and also \cite{M}.

In the proof of our main estimate in Theorem \ref{Thm1} we will use 
\eqref{MSsob} with $M=\{u=s\}$ (a level set of $u$), 
$v=|\nabla u|^{1/2}$, and $p=2$. 

The level sets of a solution $u$, and their curvature, appear in the following 
result of Sternberg and Zumbrun \cite{SZ1,SZ2}. Its statement is an inequality 
which follows from the semi-stability hypothesis \eqref{eq:02} on the solution.

\begin{proposition}[\textbf{Sternberg-Zumbrun} \cite{SZ1,SZ2}]
\label{Prop-semi-st}
Let $\Omega\subset\R^n$ be a smooth bounded domain and $u$ a smooth positive 
semi-stable solution of \eqref{problem}. Then, for every Lipschitz function 
$\eta$ in $\overline\Omega$ with $\eta|_{\partial\Omega}\equiv0$,
\begin{equation}\label{semi1}
\int_{\Omega\cap\left\{\left|\nabla u\right|>0\right\}} 
\left( |\nabla_T |\nabla u||^2 +|A|^2|\nabla u|^2\right)\eta^2\ dx
\leq \int_\Omega |\nabla u|^2 |\nabla \eta|^2 \ dx, 
\end{equation}
where $\nabla_T$ denotes the tangential or Riemannian gradient along a 
level set of $u$ {\rm (}it is thus the orthogonal projection of the full 
gradient in $\R^n$ along a level set of $u${\rm)} and where
$$
|A|^2=|A(x)|^2=\sum_{l=1}^{n-1} \kappa_l^2,
$$
being $\kappa_l$ the principal curvatures of the level set of $u$ passing
through $x$, for a given $x\in\Omega\cap\left\{\left|\nabla u\right|>0\right\}$.
\end{proposition}

This result (stated for a Neumann problem instead of a Dirichlet problem)
is Lemma 2.1 of \cite{SZ1} and Theorem 4.1 of 
\cite{SZ2}. The authors conceived and used the result to study 
qualitative properties of phase transitions in Allen-Cahn equations. 
For the sake of completeness, we give an elementary proof of it here. 
See Theorem 2.5 of \cite{FSV} for a quasilinear extension. 

\begin{proof}[Proof of Proposition \ref{Prop-semi-st}]
The semi-stability condition \eqref{eq:02} also holds, by approximation, 
for every Lipschitz function $\xi$ in $\overline{\Omega}$ with 
$\xi|_{\partial\Omega}\equiv0$. Now, take $\xi=c\eta$ in \eqref{eq:02}, 
where $c$ is a smooth function and $\eta$ Lipschitz in $\overline{\Omega}$ and 
$\eta|_{\partial\Omega}\equiv0$. A simple integration by parts gives that 
\begin{equation}\label{eq:07}
Q_u(c\eta)  = 
\int _{\Omega} c^{2}\left|\nabla \eta\right|^{2}
-\left(\Delta c+f'(u)c\right)c\eta ^{2}\ dx \geq 0.
\end{equation}
In contrast with \cite{SZ1,SZ2} (where they took 
$c=\left|\nabla u\right|$) and to avoid some considerations on the set 
$\left\{\left|\nabla u\right|=0\right\}$, we take
$$c=\sqrt{\left|\nabla u\right|^2+\varepsilon^{2}}$$
for a given $\varepsilon>0$. Note that $c$ is smooth.

Since $\Delta u+f(u)=0$ in $\Omega$, we have $\Delta u_j+f'(u)u_j=0$ in $\Omega$.
We use the notation $u_j=\partial_{x_j} u$ and also
$u_{ij}=\partial_{x_i x_j}u$. Using these equations, it is simple to verify 
that 
\begin{eqnarray*}
& &  \hspace{-1cm}\Delta c=\frac{1}{\left|\nabla u\right|^2+\varepsilon^{2}}
\left\{-f'(u)\left|\nabla u\right|^2\sqrt{\left|\nabla u\right|^2
+\varepsilon^{2}}\right.\\
&&\left.+\sum_{i,j}u_{ij}^2\sqrt{\left|\nabla u\right|^2+\varepsilon^{2}}
-\left(\sum_i\left(\sum_ju_{ij}u_j\right)^2\right)\frac{1}
{\sqrt{\left|\nabla u\right|^2+\varepsilon^{2}}}\right\},
\end{eqnarray*}
and thus
\begin{eqnarray*}
&& \hspace{-1.5cm}\left(\Delta + f'(u)\right)c=f'(u)\frac{\varepsilon^2}
{\sqrt{\left|\nabla u\right|^2+\varepsilon^{2}}}\\
&& \hspace{1cm}+\frac{1}{\sqrt{\left|\nabla u\right|^2+\varepsilon^{2}}}
\left\{\sum_{i,j}u_{ij}^2-\sum_i\left(\sum_{j}u_{ij}\frac{u_j}
{\sqrt{\left|\nabla u\right|^2+\varepsilon^{2}}}\right)^2\right\}.
\end{eqnarray*}

Using this equality in \eqref{eq:07}, we deduce 
\begin{eqnarray}\label{eq:081}
&&\hspace{-2cm}\int_{\Omega}\left(\left|\nabla u\right|^2+\varepsilon^{2}\right)
\left|\nabla \eta\right|^2dx
= \int_{\Omega}c^2
\left|\nabla \eta\right|^2dx\\
&\geq&\int_{\Omega}\left(\Delta c+f'(u)c\right)c\eta^2dx\nonumber\\
\label{eq:082}
&=&\int_{\Omega}f'(u)\varepsilon^2\eta^2dx\\
\label{eq:083}
&&+\int_{\Omega}\left\{\sum_{i,j}u_{ij}^2-
\sum_i\left(\sum_ju_{ij}\frac{u_j}{\sqrt{\left|\nabla u\right|^2
+\varepsilon^{2}}}\right)^2\right\}\eta^2dx .
\end{eqnarray}
The integrand in the last integral is nonnegative. Thus, we 
have 
\begin{eqnarray*}
&& \hspace{-1cm}\int_{\Omega}\left\{\sum_{i,j}u_{ij}^2-
\sum_i\left(\sum_ju_{ij}\frac{u_j}
{\sqrt{\left|\nabla u\right|^2
+\varepsilon^{2}}}\right)^2\right\}\eta^2dx\\
&&\geq\int_{\Omega\cap\left\{\left|\nabla u\right|>0\right\}}
\left\{\sum_{i,j}u_{ij}^2-
\sum_i\left(\sum_ju_{ij}\frac{u_j}{\sqrt{\left|\nabla u\right|^2
+\varepsilon^{2}}}\right)^2\right\}\eta^2dx\\
&&\geq\int_{\Omega\cap\left\{\left|\nabla u\right|>0\right\}}
\left\{\sum_{i,j}u_{ij}^2-\sum_i\left(\sum_ju_{ij}\frac{u_j}
{\left|\nabla u\right|}\right)^2\right\}\eta^2dx.
\end{eqnarray*}
{From} this and \eqref{eq:081}, \eqref{eq:082}, \eqref{eq:083}, we arrive at 
\begin{eqnarray*}
&&\hspace{-2cm}\int_{\Omega}\left(\left|\nabla u\right|^2
+\varepsilon^{2}\right)
\left|\nabla \eta\right|^2dx\\
&\geq&\int_{\Omega}f'(u)\varepsilon^2\eta^2dx\\
&&+\int_{\Omega\cap\left\{\left|\nabla u\right|>0\right\}}
\left\{\sum_{i,j}u_{ij}^2-\sum_i\left(\sum_ju_{ij}\frac{u_j}
{\left|\nabla u\right|}\right)^2\right\}\eta^2dx.
\end{eqnarray*}
We now let $\varepsilon\downarrow0$ to obtain 
\begin{eqnarray*}
\int_{\Omega}\left|\nabla u\right|^2
\left|\nabla \eta\right|^2dx
\geq\int_{\Omega\cap\left\{\left|\nabla u\right|>0\right\}}
\left\{\sum_{i,j}u_{ij}^2-\sum_i\left(\sum_ju_{ij}\frac{u_j}
{\left|\nabla u\right|}\right)^2\right\}\eta^2dx.
\end{eqnarray*}

We conclude the claimed inequality \eqref{semi1} of the 
proposition since
\begin{equation}\label{eq:09}
\sum_{i,j}u_{ij}^2-\sum_i\left(\sum_ju_{ij}\frac{u_j}
{\left|\nabla u\right|}\right)^2=\left|\nabla_T\left|
\nabla u\right|\right|^2+\left|A\right|^2\left|\nabla 
u\right|^2
\end{equation}
at every point $x\in\Omega\cap\left\{\left|\nabla 
u\right|>0\right\}$. This last equality can be easily checked 
assuming that $\nabla u(x)=\left(0,\ldots,0,u_{x_n}(x)\right)$ 
and looking at the quantities in \eqref{eq:09} in the orthonormal 
basis $\left\{e_1,\ldots,e_{n-1},\left(0,\ldots,0,1\right)
\right\}$, where $\left\{e_1,\ldots,e_{n-1}\right\}$ are the 
principal directions of the level set of $u$ through $x$.
See also Lemma~2.1 of \cite{SZ1} for a detailed proof of \eqref{eq:09}.
\end{proof}

Using Proposition \ref{Prop-semi-st} and Theorem \ref{ThmMS} we can now establish 
Theorem \ref{Thm1}.

\begin{proof}[Proof of Theorem \ref{Thm1}]
By elliptic regularity, the solution $u$ is smooth, that 
is, $u\in C^{\infty}(\overline\Omega)$. Recall that $u>0$ in $\Omega$.
Let us denote
$$T:=\max_{\Omega}u=\left\|u\right\|_{L^\infty (\Omega)}$$
and, for $s\in(0,T)$,
$$\Gamma_s:=\left\{x\in\Omega:u(x)=s\right\}.$$
By Sard's theorem, almost every $s\in(0,T)$ is a regular value of 
$u$. By definition, if $s$ is a regular value of $u$, then 
$\left|\nabla u(x)\right|>0$ for all $x\in\Omega$ such that 
$u(x)=s$ (i.e., for all $x\in\Gamma_s$). In particular, if $s$ is 
a regular value, $\Gamma_s$ is a $C^{\infty}$ immersed compact 
hypersurface of $\mathbb{R}^n$ without boundary. Later we will 
apply Theorem \ref{ThmMS} with $M=\Gamma_s$. Note here that, since 
$\Gamma_s$ could have a finite number of connected components, 
inequality \eqref{MSsob} for connected manifolds $M$ leads to the 
same inequality (with same constant) for $M$ with more than one 
component. 

Since $u$ is a semi-stable solution, we can use Proposition 
\ref{Prop-semi-st}. In \eqref{semi1} we take 
$$\eta(x)=\varphi(u(x))\ \ \text{for}\ \ x\in\Omega,$$ 
where $\varphi$ is a Lipschitz function in $\left[0,T\right]$ 
with
$$\varphi(0)=0.$$
The right hand side of \eqref{semi1} becomes
\begin{eqnarray*}
\int_{\Omega}\left|\nabla u\right|^2\left|\nabla\eta\right|^2dx&=&
\int_{\Omega}\left|\nabla u\right|^4\varphi'(u)^2dx\\
&=&\int_{0}^{T}\left(\int_{\Gamma_s}\left|\nabla u\right|^3dV_s
\right)\varphi'(s)^2ds,
\end{eqnarray*}
by the coarea formula. We have denoted by $dV_s$ the volume element 
in $\Gamma_s$. The integral in $ds$ is over the regular values 
of $u$, whose complement is of zero measure in $\left(0,T\right)$.

In the left hand side of \eqref{semi1} we integrate only on 
$\Omega\cap\left\{\left|\nabla u\right|>\delta\right\}$ for a given 
$\delta>0$, and thus inequality \eqref{semi1} remains valid. Since 
in this set $|\nabla u|$ is bounded away from zero, the coarea 
formula gives

\begin{eqnarray*}
&&\hspace{-1cm}\int_{0}^{T}\left(\int_{\Gamma_s}\left|\nabla u\right|^3dV_s
\right)\varphi'(s)^2ds\\
&&\geq\int_{\Omega\cap\left\{\left|\nabla 
u\right|>\delta\right\}}\left(\left|\nabla_T\left|
\nabla u\right|\right|^2+\left|A\right|^2\left|\nabla 
u\right|^2\right)\varphi(u)^2dx\\
&&=\int_{0}^{T}\left(\int_{\Gamma_s\cap\left\{\left|\nabla 
u\right|>\delta\right\}}\frac{1}{\left|\nabla u\right|}
\left(\left|\nabla_T\left|
\nabla u\right|\right|^2+\left|A\right|^2\left|\nabla 
u\right|^2\right)dV_s\right)\varphi(s)^2ds\\
&&=\int_{0}^{T}\left(\int_{\Gamma_s\cap\left\{\left|\nabla 
u\right|>\delta\right\}}
4\left|\nabla_T\left|
\nabla u\right|^{1/2}\right|^2+\left(\left|A\right|\left|\nabla 
u\right|^{1/2}\right)^{2}dV_s\right)\varphi(s)^2ds .
\end{eqnarray*}
Letting $\delta\downarrow0$ and using the monotone convergence theorem, 
we deduce that

\begin{equation}
\int_0^T h_1(s) \varphi(s)^2 \ ds
\leq
\int_0^T h_2(s)  \varphi'(s)^2 \ ds,
\label{semi3}
\end{equation}
for all Lipschitz functions 
$\varphi:\left[0,T\right]\rightarrow\mathbb{R}$ with 
$\varphi(0)=0$, where 
\begin{equation}\label{eq:27}
h_1(s):=\int_{\Gamma_s} 
 4|\nabla_T |\nabla u|^{1/2}|^2 +\left( |A||\nabla u|^{1/2} \right)^2\ dV_s
\end{equation}
and
\begin{equation}\label{eq:28}
h_2(s):=\int_{\Gamma_s} |\nabla u|^3\  dV_s 
\end{equation}
for every regular value $s$ of $u$. 

Inequality \eqref{semi3}, with $h_1$ and $h_2$ as defined above, will lead 
to our $L^{\infty}$ estimate of Theorem \ref{Thm1} after choosing an 
appropriate test function $\varphi$ in \eqref{semi3}. In dimensions 2 and 3 
we will simply use \eqref{semi3} and well known geometric inequalities 
about the curvature of manifolds (note that $h_1$ involves the curvature 
of the level sets of $u$). Instead, in dimension 4 we need the following 
additional tool. For $n\geq4$ we use the Michael-Simon and Allard 
Sobolev inequality \eqref{MSsob} with $M=\Gamma_s$, $p=2<m=n-1$, and 
$v= \left|\nabla u\right|^{1/2}$. Note that the mean curvature $H$ of 
$\Gamma_s$ satisfies $\left|H\right|\leq\left|A\right|$. We obtain
\begin{equation}\label{estimate2}
\left( \int_{\Gamma_s} |\nabla u|^\frac{n-1}{n-3} \ dV_s \right)^\frac{n-3}
{n-1}\leq C(n) h_1(s)
\end{equation}
for all regular values $s$ of $u$,  where $C(n)$ us a constant depending 
only on $n$. This estimate combined with \eqref{semi3} leads to 
\begin{equation}\label{estgenn}
\int_{0}^{T}\left(\int_{\Gamma_s} |\nabla u|^\frac{n-1}{n-3} \ dV_s 
\right)^\frac{n-3}{n-1}\varphi(s)^2ds\leq C(n)\int_{0}^{T}\left(\int_{\Gamma_s} 
|\nabla u|^3 \ dV_s \right)\varphi'(s)^2ds
\end{equation}
for all Lipschitz functions $\varphi$ in $[0,T]$ with $\varphi(0)=0$. We only know how 
to derive an $L^{\infty}$ estimate for $u$ (i.e., a bound on 
$T=\max_{\Omega}u$) from \eqref{estgenn} when the exponent $(n-1)/(n-3)$ on 
its left hand side is larger or equal than the one on the right hand side, 
i.e., 3. That is, we need $(n-1)/(n-3)\geq3$, which means $n\leq4$.

The rest of the proof differs in every dimension $n=4,3,$ and 2. But in  
all three cases it will be useful to denote 
\begin{equation}\label{defst}
B_t:=\frac{1}{t^2}\int_{\left\{u<t\right\}}\left|\nabla u\right|^4dx =
\frac{1}{t^2}\int_0^th_2(s)ds,
\end{equation}
where $t>0$ is a given positive constant as in the statement of the 
theorem. Note that the quantity $B_t$ is the main part of the right hand 
side of our estimate \eqref{est4}. Let us start with the

\smallskip

\underline{\it Case $n=4$}. 
It will be crucial to use the bound obtained in \eqref{estimate2}. It reads,
since $(n-1)/(n-3)=3$,
\begin{equation}\label{cont4}
h_2^{1/3}\leq C h_1 \qquad\text{a.e. in}\ (0,T)\ (\text{when}\ n=4),
\end{equation}
where $C$ is a universal constant. For every regular value $s$ of $u$, we 
have $0<h_2(s)$ and $h_1(s)<\infty$ (simply by their definition). 
This together with \eqref{cont4} gives $h_1/h_2\in\left(0,+\infty\right)$ 
a.e. in $\left(0,T\right)$. Thus, defining 
$$g_k(s):=\min\left\{k,\frac{h_1(s)}{h_2(s)}\right\}$$
for regular values $s$ and for a positive integer $k$, we have that 
$g_k\in L^{\infty}(0,T)$ and 
\begin{equation}\label{gknicr}
g_k(s)\nearrow\frac{h_1(s)}{h_2(s)}\in(0,+\infty) \qquad\text{as}\ 
k\uparrow\infty,\ \text{for}\ \text{a.e.}\ s\in(0,T).
\end{equation}
Since $g_k\in L^{\infty}(0,T)$, the function 
\begin{equation}\label{eq:211}
\varphi_k(s):=\left\{
\begin{array}{lll}
s/t&\textrm{if}&s\leq t,\\
\displaystyle \exp\left(
\frac{1}{\sqrt{2}}\int_t^s \sqrt{g_k(\tau)}\ d\tau\right)
&\textrm{if}&t<s\leq T,
\end{array}
\right.
\end{equation}
is well defined, Lipschitz in $\left[0,T\right]$, and satisfies 
$\varphi_k(0)=0$.

Since
$$h_2(\varphi_k')^2=h_2\frac{1}{2}g_k\varphi_k^2\leq\frac{1}{2}
h_1\varphi_k^2\qquad \text{in}\ (t,T),$$
\eqref{semi3} used with $\varphi=\varphi_k$ leads to
\begin{equation}\label{eq:213}
\int_t^T h_1\varphi_k^2 ds\leq \frac{2}{t^2}\int_0^t h_2\  ds=
\frac{2}{t^2}\int_{\{u< t\}}|\nabla u|^4\  dx=2 B_t.
\end{equation}
Recall that $B_t$ was defined in \eqref{defst} and that we 
need to establish $T-t\leq C B_t^{1/2}$. By \eqref{gknicr} 
we have
\begin{equation}\label{eq:214}
T-t=\int_t^Tds=\sup_{k\geq1}\int_t^T\sqrt[4]{\frac{h_2g_k}{h_1}}ds.
\end{equation}
Using \eqref{eq:213} and Cauchy-Schwarz, we have that
\begin{eqnarray}\label{eq:218}
\int_t^T\sqrt[4]{\frac{h_2g_k}{h_1}}ds&=&\int_t^T\left(\sqrt{h_1}
\varphi_k\right)\left(\sqrt[4]{\frac{h_2g_k}{h_1^3}}\frac{1}
{\varphi_k}\right)ds\\
&\leq& \left(2B_t\right)^{1/2}\left\{\int_t^T\sqrt{\frac{h_2g_k}{h_1^3}}
\frac{1}{\varphi_k^2}ds\right\}^{1/2} \nonumber \\ 
\label{2.20}
&\leq& \left(2B_t\right)^{1/2}\left\{C\int_t^T\sqrt{g_k}
\frac{1}{\varphi_k^2}ds\right\}^{1/2}. 
\end{eqnarray}
In the last inequality we have used our crucial estimate \eqref{cont4}.

Finally we bound the integral in \eqref{2.20}, using the definition 
\eqref{eq:211} of $\varphi_k$, as follows:
\begin{eqnarray*}
\int_t^T\sqrt{g_k}\frac{1}{\varphi_k^2}ds &=& \int_t^T\sqrt{g_k}
\frac{1}{\varphi_k^2}\frac{\varphi_k'}{\frac{1}{\sqrt{2}}
\sqrt{g_k}\varphi_k}ds\\
&=& \sqrt{2}\int_t^T\frac{\varphi'_k}{\varphi_k^3}ds = 
\frac{\sqrt{2}}{2}\left[\varphi_k^{-2}(s)\right]_{s=T}^{s=t} \\
&\leq&\frac{\sqrt{2}}{2}\varphi_k^{-2}(t)=\frac{\sqrt{2}}{2}.  
\end{eqnarray*}
This bound together with \eqref{eq:214},\eqref{eq:218}, and \eqref{2.20} finish 
the proof in dimension $n=4$. 

Let us now turn to the 

\smallskip

\underline{{\it Cases} $n=3$ {\it and} 2}. In these dimensions we take a simpler 
test function $\varphi$ in \eqref{semi3} than in dimension 4. We 
simply consider 
\begin{equation*}\label{eq:212}
\varphi(s)=\left\{
\begin{array}{lll}
s/t&\textrm{if}&s\leq t\\
1 &\textrm{if}&t<s.
\end{array}
\right.
\end{equation*}
With this choice of $\varphi$ and since 
$h_1(s)\geq\int_{\Gamma_s}\left|A\right|^2\left|\nabla u\right|dV_s\ $ 
---see definition \eqref{eq:27}---, inequality \eqref{semi3} leads to
\begin{equation}\label{eq:224}
\begin{split}
& \hspace{-2cm}\int_t^T\int_{\Gamma_s}
\left|A\right|^2\left|\nabla u\right|dV_sds
\leq\int_0^Th_1(s)\varphi(s)^2ds\\
& \leq \int_0^th_2(s)\frac{1}{t^2}ds = \frac{1}{t^2}
\int_{\left\{u<t\right\}}\left|\nabla u\right|^4dx=:B_t.
\end{split}
\end{equation}
This inequality and the ones that follow hold in 
every dimension $n$. It is at the end of the proof that 
we will need to assume $n\leq3$. 

Next, we use a well known geometric inequality for the curve 
$\Gamma_s$ ($n=2$) or the surface $\Gamma_s$ ($n=3$). It also 
holds in every dimension $n\geq2$ and it states
\begin{equation}\label{eq:225}
\left|\Gamma_s\right|^{\frac{n-2}{n-1}}\leq C(n) \int_{\Gamma_s}
\left|H\right|dV_s,
\end{equation}
where $H$ is the mean curvature of $\Gamma_s$, $C(n)$ is a constant 
depending only on $n$, and $s$ is a regular value of $u$. In 
dimension $n=2$ this simply follows from the Gauss-Bonnet formula. 
For $n\geq3$, \eqref{eq:225} is stated in Theorem 28.4.1 of 
\cite{BZ} and follows from the Michael-Simon and Allard Sobolev 
inequality (Theorem \ref{ThmMS} of our paper). Indeed, taking $v\equiv 1$ 
and $m=n-1>1=p$ in \eqref{MSsob}, we deduce \eqref{eq:225}. Note 
that \eqref{eq:225} also holds if $\Gamma_s$ is not connected 
(with the same constant $C(n)$ as for connected manifolds).

We also use the classical isoperimetric inequality,
\begin{equation}\label{eq:226}
V(s):=\left|\left\{u>s\right\}\right|\leq 
C(n)\left|\Gamma_s\right|^{\frac{n}{n-1}},
\end{equation}
which also holds, with same constant $C(n)$, in case 
$\left\{u>s\right\}$ is not connected. Now, \eqref{eq:225} and 
\eqref{eq:226} lead to 
\begin{equation*}
V(s)^{\frac{n-2}{n}}\leq C(n) \int_{\Gamma_s}\left|H\right|dV_s
\leq C(n)\left\{\int_{\Gamma_s}\left|A\right|^2\left|\nabla u\right| 
dV_s\right\}^{1/2}\left\{\int_{\Gamma_s}\frac{dV_s}
{\left|\nabla u\right|}\right\}^{1/2}
\end{equation*}
for all regular values $s$,
by Cauchy-Schwarz and since $\left|H\right|\leq\left|A\right|$. 
{From} this, we deduce 
\begin{eqnarray}\label{eq:227}
T-t=\int_t^T ds&\leq& \int_t^T C(n)\left\{\int_{\Gamma_s}
\left|A\right|^2\left|\nabla u\right| 
dV_s\right\}^{1/2}\cdot\\
\nonumber
&&\hspace{6mm}\cdot\left\{V(s)^{\frac{2(2-n)}{n}}
\int_{\Gamma_s}\frac{dV_s}
{\left|\nabla u\right|}\right\}^{1/2} ds\\
\nonumber
&\leq&C(n)\left\{\int_t^T \int_{\Gamma_s}
\left|A\right|^2\left|\nabla u\right| 
dV_sds\right\}^{1/2}\cdot\\
\nonumber
&&\hspace{6mm}\cdot\left\{\int_t^T V(s)^{\frac{2(2-n)}{n}}
\int_{\Gamma_s}\frac{dV_s}
{\left|\nabla u\right|}ds\right\}^{1/2}\\
\label{eq:227bis}
&\leq&C(n)B_t^{1/2}\left\{\int_t^T
V(s)^{\frac{2(2-n)}{n}}\int_{\Gamma_s}\frac{dV_s}
{\left|\nabla u\right|}ds\right\}^{1/2},
\end{eqnarray}
where we have used \eqref{eq:224} in the last inequality. 

Finally, since $V(s)=\left|\left\{u>s\right\}\right|$ is a 
nonincreasing function, it is differentiable almost everywhere 
and, by the coarea formula, 
\begin{equation*}
-V'(s)=\int_{\Gamma_s}\frac{dV_s}{\left|\nabla u\right|}
\qquad \text{for a.e.}\ s\in(0,T).
\end{equation*}
In addition, for $n\leq3$, $V(s)^{\frac{4-n}{n}}$ is 
nonincreasing in $s$ and thus its total variation satisfies
\begin{eqnarray*}
\left|\Omega\right|^{\frac{4-n}{n}}&\geq& V(t)^{\frac{4-n}{n}}
=\left[V(s)^{\frac{4-n}{n}}\right]^{s=t}_{s=T}\\
&\geq&\int_t^T \frac{4-n}{n} V(s)^{\frac{2(2-n)}{n}}
\left(-V'(s)\right)ds \\
&=&\frac{4-n}{n}\int_t^T V(s)^{\frac{2(2-n)}{n}}\int_{\Gamma_s}
\frac{dV_s}{\left|\nabla u\right|}ds.
\end{eqnarray*}
{From} this, \eqref{eq:227}, and \eqref{eq:227bis}, we conclude the desired inequality 
\begin{equation}
T-t\leq C(n) B_t^{1/2} \left|\Omega\right|^{(4-n)/(2n)},
\end{equation}
for $n\leq3$.

Note that this argument gives nothing for $n\geq 4$ since the
integral in \eqref{eq:227bis}, 
\begin{equation}
\int_t^T V(s)^{\frac{2(2-n)}{n}}\left(-V'(s)\right)ds 
= \int_0^{V(t)} \frac{dr}{r^{\frac{2(n-2)}{n}}},
\end{equation}
is not convergent at $s=T$ (i.e., $r=0$) because $2(n-2)/n\geq1$.
\end{proof}

\section{Proof of Theorems \ref{Thm3} and \ref{Thm2}}

In this last section we establish Theorem \ref{Thm3} and, as a 
simple consequence, Theorem \ref{Thm2}. They will follow easily 
from the following proposition. It states that, thanks to 
Theorem \ref{Thm1}, an $L^{\infty}(\Omega)$ estimate 
for a semi-stable solution follows from having an $L^{\infty}$ bound for the 
solution near the boundary of $\Omega$.

\begin{proposition}\label{prop}
Let $f$ be any $C^\infty$ function.
Let $\Omega\subset \mathbb{R}^n$ be any $C^\infty$ bounded domain.
Assume that $2\leq n\leq 4$.

Let $u$
be a classical semi-stable solution of \eqref{problem}. Assume that
\begin{equation}\label{dist}
u\geq c_3 \,\text{\rm dist} (\cdot,\partial \Omega) \ \ \text{in } \Omega
\end{equation}
and
\begin{equation}\label{linfK}
\Vert u\Vert_{L^\infty(\Omega_\varepsilon)}\leq c_4,
\ \ \text{ where }  \Omega_\varepsilon=\{x\in\Omega\, :\, \text{\rm dist}(x,\partial\Omega)
<\varepsilon\},
\end{equation}
for some positive constants $\varepsilon$, $c_3$, and $c_4$.

Then,
\begin{equation}\label{linf34}
\Vert u\Vert_{L^\infty(\Omega)} \leq C\left(\Omega,\varepsilon,
c_3,c_4,\Vert f\Vert_{L^\infty(\,[0,c_4]\,)} \right),
\end{equation}
where $C(\cdot)$ is a constant depending only on the
quantities within the parentheses.
\end{proposition}

\begin{proof}
By taking $\varepsilon$ smaller if necessary, we may assume that
$\Omega_\delta=\{x\in\Omega\, :\, \text{\rm dist}(x,\partial\Omega)
<\delta\}$ is $C^\infty$ for every $0<\delta<\varepsilon$.

We use Theorem~\ref{Thm1} with the choice 
$$
t=c_3\frac{\varepsilon}{2}.
$$
By \eqref{dist}, the set $\{u<t\}$ in the right hand side of our main
estimate, \eqref{est4},  satisfies
$$
\{u<t\}\subset \Omega_{\varepsilon/2}.
$$
Thus, it suffices to bound $\Vert u\Vert_{W^{1,4}(\Omega_{\varepsilon/2})}$.

But $u$ is a solution of $-\Delta u= f(u)$ in $\Omega_\varepsilon$ and
$u=0$ on $\partial\Omega$ (which is one part of $\partial\Omega_\varepsilon$).
On the other hand, $\partial\Omega \cup \Omega_{\varepsilon/2}$ 
has compact closure 
contained in $\partial\Omega \cup \Omega_\varepsilon$, and both sets
are $C^\infty$.
By \eqref{linfK}, 
$\Vert u\Vert_{L^\infty(\Omega_\varepsilon)}\leq c_4$ and thus the right
hand side of the equation satisfies
$\Vert f(u)\Vert_{L^\infty(\Omega_\varepsilon)}\leq \Vert f\Vert_{L^\infty(\, 
[0,c_4]\, )}$. Hence, by interior and boundary estimates for the linear 
Poisson equation,
we deduce a bound $\Vert u\Vert_{W^{1,4}(\Omega_{\varepsilon/2})}$ depending
on the quantities in \eqref{linf34}. 
\end{proof}

The $L^\infty$ bound \eqref{linfK} in a neighborhood of $\partial \Omega$ is known to 
hold for every nonlinearity $f$ when $\Omega$ is a convex domain
(in every dimension $n\geq 2$). This is 
proved using the moving planes method and holds for every positive solution 
---not only for semi-stable solutions. The precise statement is the following.

\begin{proposition}[\cite{GNN, FLN, CL}]\label{Prop4}
Let $f$ be any locally Lipschitz function and let $\Omega\subset\R^{n}$
be a $C^\infty$ bounded domain. Let $u$ be any 
positive classical solution of \eqref{problem}.

If $\Omega$ is convex, then there exist positive constants $\rho$ 
and $\gamma$ depending only on the domain $\Omega$ such that
for every $x\in\Omega$ with $\text{\rm dist}(x,\partial\Omega)<\rho$,
there exists a set $I_x$ with the following properties: 
\begin{equation}\label{ix}
|I_x|\geq\gamma \qquad\text{and}\qquad
u(x) \leq u(y) \ \text{ for all }  y\in I_x.
\end{equation}
As a consequence,
\begin{equation}\label{lix}
\Vert u\Vert_{L^\infty(\Omega_\rho)}\leq \frac{1}{\gamma} \Vert u\Vert_{L^1 (\Omega)},
\ \ \text{ where }  \Omega_\rho=\{x\in\Omega\, :\, \text{\rm dist}(x,\partial\Omega)
<\rho\}.
\end{equation}

If $\Omega$ is not convex but we assume $n=2$ and $f\geq0$, then 
\eqref{lix} also holds for some constants $\rho$ and $\gamma$ 
depending only on $\Omega$.
\end{proposition}

The proof of this proposition uses the moving planes method of 
Gidas-Ni-Nirenberg~\cite{GNN}. Assume that $\Omega$ is $C^\infty$ 
and convex, $n\geq 2$. For $y\in\partial\Omega$, let $\nu (y)$ be 
the unit outward normal to $\Omega$ at $y$.
There exist positive constants $s_0$ and $\alpha$ depending only
on the convex domain $\Omega$ such that, for every $y\in\partial\Omega$
and every $e\in\R^n$ with $|e|=1$ and $e\cdot \nu (y)\geq\alpha$,
we have that $u(y - se)$ is nondecreasing in $s\in [0, s_0]$. This fact 
follows from the reflection method applied
to planes close to those tangent to $\Omega$ at 
$\partial \Omega$. By the convexity of $\Omega$, the reflected caps 
will be contained in $\Omega$. The previous monotonicity fact
leads to \eqref{ix}, where $I_x$ is a truncated open cone with vertex at $x$. 
If all curvatures of $\partial\Omega$ 
are positive this is quite simple to prove and, as mentioned in page 45 
of de Figuereido-Lions-Nussbaum \cite{FLN}, can also be proved for 
convex domains with a little more of care. 

In \cite{FLN} it is also proved that the boundary estimate \eqref{lix} 
holds also for general (nonconvex) smooth domains $\Omega$ if the 
nonlinearity $f$ is subcritical in the sense that
\begin{equation}\label{subcr}
f(t)t^{-\frac{n+2}{n-2}}\ \text{is nonincreasing in}\ t\in [0,+\infty),
\end{equation}
when $n\geq3$. For these nonlinearities, we do not need to assume the 
convexity of $\Omega$ in our results. This result is proved with 
the aid of some Kelvin transforms ---after which one can use the 
moving planes method; see \cite{FLN}. 

When $n=2$, Chen and Li \cite{CL} use this Kelvin transform method 
to establish the boundary estimate \eqref{lix} in nonconvex domains 
$\Omega\subset\mathbb{R}^2$ assuming only $f\geq0$ ---as stated 
at the end of Proposition \ref{Prop4}.

Using Propositions \ref{prop} and \ref{Prop4}, we can now give the

\begin{proof}[Proof of Theorem \ref{Thm3}]
We use Proposition~\ref{prop}.
We assume $f\geq 0$ ---and
also $\Omega$ convex in case $n\in\{3,4\}$. Let $u_k$ be a 
sequence of classical positive semi-stable solutions of 
\eqref{problem} converging to $u$ in $L^1(\Omega)$. 

For $x\in\Omega$ and $v:\Omega\to\R$, denote
$$
\delta (x)= \text{\rm dist}(x,\partial\Omega) \qquad 
\text{and}\qquad \Vert v\Vert_{L^1_\delta(\Omega)}=
\Vert v\delta \Vert_{L^1(\Omega)}.
$$
By Proposition~\ref{Prop4},
\begin{equation}\label{l1}
\Vert u_k\Vert_{L^\infty(\Omega_\rho)}\leq \frac{1}{\gamma} \Vert 
u_k\Vert_{L^1 (\Omega)} \longrightarrow \frac{1}{\gamma} \Vert 
u\Vert_{L^1 (\Omega)},
\end{equation}
as $k\to\infty$, where $\rho$ and $\gamma$ are positive constants
depending only on $\Omega$.

Next, since $f\geq 0$, we can use a simple estimate for the linear Poisson equation
$-\Delta u_k=h_k(x):=f(u_k(x))\geq 0$ with zero Dirichlet boundary conditions.
It states that
\begin{equation}\label{lower}
\frac{u_k}{\delta}\geq  c\Vert f(u_k)\Vert_{L^1_\delta (\Omega)}\qquad\text{in }
\Omega,
\end{equation}
for some positive constant $c$ depending only on $\Omega$ ---see for
instance Lemma~3.2 of \cite{BC} for a simple proof.

Multiply \eqref{problem} (with $u$ replaced by $u_k$) 
by the first Dirichlet eigenfunction of $-\Delta$
in $\Omega$ and integrate twice by parts. We deduce that 
$\Vert u_k\Vert_{L^1_\delta (\Omega)}$ and $\Vert f(u_k)\Vert_{L^1_\delta (\Omega)}$ 
are comparable up to multiplicative constants depending only on $\Omega$.
Multiplying \eqref{problem} now by the solution $w$ of
\begin{equation}\label{1rhs}
\left\{
\begin{array}{rcll}
-\Delta w &=&1&\textrm{in }\Omega\\
w&=&0&\textrm{on }\partial \Omega,
\end{array}\right.
\end{equation}
we deduce
that also $\Vert u_k\Vert_{L^1 (\Omega)}$
is comparable to the two previous quantities.
Recall that $\Vert u_k\Vert_{L^1 (\Omega)}\to \Vert u\Vert_{L^1 (\Omega)}>0$. 
Hence, the right hand side of \eqref{lower} is bounded below by a positive 
constant independent of $k$. 

As a consequence of this lower bound and of \eqref{l1}, Proposition~\ref{prop}
gives a uniform $L^\infty(\Omega)$ estimate for all solutions $u_k$.
Letting $k\to\infty$ we deduce $u\in L^\infty(\Omega)$, as claimed in
part (i) of the theorem.

To prove part (ii), we simply make more precise the constants
in \eqref{l1} and \eqref{lower}. Since we now assume $f\geq c_1>0$, we have
that $u_k\geq c_1w \geq c_1c\,\delta=c_1c\,\text{\rm dist}(\cdot,\partial\Omega)$ 
in $\Omega$, where $w$ is the solution of \eqref{1rhs} and $c$ depends only on~$\Omega$.
This is estimate \eqref{dist} of Proposition~\ref{prop}.

Finally, 
we multiply \eqref{problem} (with $u$ replaced by $u_k$) 
by the first Dirichlet eigenfunction of $-\Delta$
in $\Omega$ and integrate twice by parts. Using that $f(s) \geq\mu s-c_2$
for all $s$ and that $\mu >\lambda_1$, we obtain a control
$\Vert u_k \Vert_{L^1_\delta (\Omega)}\leq \overline C=\overline C(\Omega,\mu, c_2)$ 
and thus 
also for $\left\|u_k\right\|_{L^1(\Omega)}$ as mentioned before. 
Now, this estimate combined with \eqref{l1} give an estimate
as \eqref{linfK}. Estimate \eqref{linf34} in Proposition~\ref{prop} gives the desired
conclusion \eqref{linf} of Theorem~\ref{Thm3}.
\end{proof}

Theorem \ref{Thm2} on the boundedness of the extremal solution $u^*$ follows 
easily from Theorem \ref{Thm3}.

\begin{proof}[Proof of Theorem \ref{Thm2}]
We extend $g$ in a $C^1$ manner to all of $\R$ with $g$ nondecreasing 
and $g\geq g(0)/2$ in $\R$. 
Recall the the extremal (weak) solution $u^*$ is the 
increasing $L^1$ limit, as $\lambda\uparrow\lambda^*$, of the minimal
solutions $u_\lambda$ of $(\pbla)$. In addition, for $\lambda <\lambda^*$,
$u_\lambda$ is a $C^2$ semi-stable solution of $(\pbla)$ ---see Remark 
\ref{rem-semi} in the introduction.

If $g$ is $C^\infty$ we simply apply part (ii) of Theorem~\ref{Thm3} with $f=\lambda g$
for $\lambda^*/2 < \lambda <\lambda^*$. Using that $g$ satisfies
\eqref{nonlg}, we can verify \eqref{assf} and obtain estimates for
$\Vert u_\lambda\Vert_{L^\infty(\Omega)}$ which are uniform in $\lambda$.
Letting $\lambda\uparrow\lambda^*$ we conclude that $u^*\in L^\infty(\Omega)$.

In case that $g\in C^1$ is not $C^\infty$, let $\rho_k$ be a $C^\infty$ mollifier 
with support in $(0,1/k)$, of the form $\rho_k(\beta)=k\rho(k\beta)$.
We replace $g$ by
$$
g_k (s) =\int_{s-1/k}^s g(\tau) \rho_k (s-\tau) \ d\tau
=\int_{0}^1 g(s-\beta /k) \rho (\beta) \ d\beta .
$$
For all $k$, we have that $g_k\leq g_{k+1}\leq g$ in~$\R$,
$g_k$ is $C^\infty$, and (as $g$) nondecreasing. In addition, $g_k$ satisfies all conditions
in \eqref{nonlg}. Since $g(u^*)\geq g_k(u^*)$, $u^*$ is a
weak supersolution for problem $(\pblaest)$ with $g$ replaced
by $g_k$. By the monotone iteration procedure, it 
follows that the extremal
parameter for $g_k$,  $\lambda^*_k$, satisfies $\lambda^*\leq\lambda^*_k$.
Hence $u^k_{\lambda^*-1/k}$, the solution for problem $(\pblaest)$ 
with $g$ replaced
by $g_k$ and with $\lambda=\lambda^*-1/k$ is classical.
Thus, we can apply Theorem~\ref{Thm3} with $f=\lambda g_k$ and
$\lambda=\lambda^*-1/k$ to obtain an $L^{\infty}(\Omega)$ bound for 
$u_{\lambda^*-1/k}^k$ independent of~$k$. 
Note that $u_{\lambda^*-1/k}^{k} \leq u_{\lambda^*-1/(k+1)}^{k}$ and that,
since $g_k\leq g_{k+1}\leq g$,  $u_{\lambda^*-1/(k+1)}^{k}\leq u_{\lambda^*-1/(k+1)}^{k+1}
\leq u_{\lambda^*}=u^*$. Thus, 
$u^k_{\lambda^*-1/k}$ 
increases in $L^1(\Omega)$ towards a solution of $(\pblaest)$ smaller or equal
than $u^*$, and hence identically $u^*$.  {From} the $L^{\infty}(\Omega)$ bound for 
$u_{\lambda^*-1/k}^k$ independent of $k$, we conclude 
$u^*\in L^\infty(\Omega)$.
\end{proof}

\noindent
\textbf{Acknowledgment.} The author would like to thank Manel 
Sanch\'on for simplifying the proof of Theorem \ref{Thm1} in 
dimension $n=3$. His simpler proof is the one given in this paper.

\end{document}